\documentclass[11pt,reqno]{amsart}
\usepackage{graphicx,amsmath,amsthm,verbatim,amssymb,lineno,titletoc}
\usepackage{mathrsfs}
\usepackage{adjustbox}
\usepackage{ytableau}
\usepackage{bbm}
\usepackage{hyperref}
\hypersetup{colorlinks}
\usepackage[numbers]{natbib}
\usepackage{array}
\usepackage{tikz}
\usetikzlibrary{shapes.multipart,patterns,arrows}

\usepackage[font=footnotesize]{caption}

\usepackage{amssymb}
\usepackage[T1]{fontenc} 
\usepackage{fourier} 
\usepackage[english]{babel} 
\usepackage{amsmath,amsfonts,amsthm,bbm} 
\usepackage{appendix}

\usepackage{amssymb}
\usepackage[all]{xy}
\usepackage{enumerate}
\usepackage{tikz}
\usepackage{mathrsfs}
\usepackage[english]{babel}
\usepackage{multirow}
\usepackage{float}
\usepackage{enumitem}
\usepackage{graphicx,accents}

\newcolumntype{M}[1]{>{\centering\arraybackslash}m{#1}}
\newcolumntype{N}{@{}m{0pt}@{}}

\pdfoutput=1

\newtheorem{theorem}{Theorem}[section]
\newtheorem{prop}[theorem]{Proposition}

\newtheorem{lemma}[theorem]{Lemma}

\def\Z{\mathbb{Z}}

\def\R{\mathbb{R}}

\def\EE{\mathbb{E}}
\def\PP{\mathbb{P}}

\def\L{\mathcal{L}}

\def\bZ{\mathbb{Z}}
\def\tS{\mathtt{S}}

\newcommand{\indicator}{{\mathbbm 1}}

\newtheorem{innercustomthm}{Theorem}
\newenvironment{customthm}[1]
{\renewcommand\theinnercustomthm{#1}\innercustomthm}
{\endinnercustomthm}

\newtheorem{innercustomlemma}{Lemma}
\newenvironment{customlemma}[1]
{\renewcommand\theinnercustomlemma{#1}\innercustomlemma}
{\endinnercustomlemma}

\newtheorem{innercustomprop}{Proposition}
\newenvironment{customprop}[1]
{\renewcommand\theinnercustomprop{#1}\innercustomprop}
{\endinnercustomprop}

\theoremstyle{definition}
\newtheorem{innercustomexample}{Example}
\newenvironment{customexample}[1]
{\renewcommand\theinnercustomexample{#1}\innercustomexample}
{\endinnercustomexample}

\newtheorem{innercustomremark}{Remark}
\newenvironment{customremark}[1]
{\renewcommand\theinnercustomremark{#1}\innercustomremark}
{\endinnercustomremark}

\renewcommand{\1}{\mathbbm{1}}

\def\L{\mathbf{l}}
\def\R{\mathbf{r}}

\def\nbe{\mathtt{ne}}

\DeclareRobustCommand{\cev}[1]{\reflectbox{\ensuremath{\vec{\reflectbox{\ensuremath{#1}}}}}}

\allowdisplaybreaks

\DeclareRobustCommand{\cev}[1]{\reflectbox{\ensuremath{\vec{\reflectbox{\ensuremath{#1}}}}}}

\begin{document}
	
	\title[Persistence of random walks with correlated increments]{Persistence of sums of correlated increments \\
	and clustering in cellular automata	
}

	\author{Hanbaek Lyu}
	\address{Hanbaek Lyu, Department of Mathematics, The Ohio State University, Columbus, OH 43210.}
	\email{\texttt{colourgraph@gmail.com}}
	
	\author{David Sivakoff}
	\address{David Sivakoff, Department of Statistics and Mathematics, The Ohio State University, Columbus, OH 43210}
	\email{\texttt{dsivakoff@stat.osu.edu}}
	
	\date{\today}
	

	\begin{abstract}
	We consider sums of increments given by a functional of a stationary Markov chain. Letting $T$ be the first return time of the partial sums process to $(-\infty,0]$, under general assumptions, we determine the asymptotic behavior of the survival probability, $\mathbb{P}(T\ge t)\sim Ct^{-1/2}$ for an explicit constant $C$. Our analysis is based on a connection between the survival probability and the running maximum of the time-reversed process, and relies on a functional central limit theorem for Markov chains. Our result extends the classic theorem of Sparre Anderson on sums of mean zero and independent increments to the case of correlated increments. As applications, we recover known clustering results for the 3-color cyclic cellular automaton and the Greenberg-Hastings model in one dimension, and we prove a new clustering result for the 3-color firefly cellular automaton.

	\end{abstract}

	\maketitle

	\section{Introduction}
	\label{Introduction}

	Let $(\mathtt{S}_{t})_{t\in \bZ_+}$ be a discrete-time real valued stochastic process with $\mathtt{S}_{0}=r$, and for now let us assume $r=0$. For 
	$\ell\le0$, let $T_{\ell}=\inf\{ t\ge 1\,:\, \mathtt{S}_{t}\le \ell \}$ be the first passage time to $(-\infty,\ell]$. The \textit{survival probability} $\mathbb{P}(T_{\ell}\ge t)$, or the \textit{persistence} of the process, has been extensively studied as a means of understanding non-equilibrium systems such as Ising or Potts models, diffusion equations and fluctuating interfaces (for surveys, see \cite{bray2013persistence}, \cite{aurzada2015persistence}).    
	
	The classical setting is to take $\mathtt{S}_{t}$ to be the sum of $t$ i.i.d. increments. A cornerstone result, known as Sparre-Anderson's formula (see e.g., \cite{andersen1953sums} and Theorem XII.7.1 in \cite{feller1971introduction}), gives the generating function of the probabilities $\mathbb{P}(T_{0}\ge t)$ in terms of that of $t^{-1}\mathbb{P}(\mathtt{S}_{t}\le 0)$. That is, for all $u\in [0,1)$, 
	\begin{equation}\label{eq:sparre_anderson}
	F(u):=\sum_{t\ge 0} \mathbb{P}(T_{0}\ge t) u^{t} = \exp\left( \sum_{t\ge 1} \mathbb{P}(\mathtt{S}_{t}\le 0)\frac{u^{t}}{t} \right).
	\end{equation}
	A stronger form of (\ref{eq:sparre_anderson}) is known as Spitzer's identity \cite{spitzer1956combinatorial}.
	
	If $\mathbb{P}(\mathtt{S}_{t}\le 0)=\rho \in (0,1)$ for all large enough $t$, then (\ref{eq:sparre_anderson}) gives $F(u)\sim (1-u)^{\rho}$ as $u\rightarrow 1^{-}$, so a Tauberian theorem for monotonic sequences (Theorem XIII.5.5 in \cite{feller1971introduction}) gives 
	\begin{equation}\label{eq:i.i.d. increments with infinite variance}
	\mathbb{P}(T_{0}\ge t)  \sim t^{\rho-1}/\Gamma(\rho).
	\end{equation}
	In the case of the simple symmetric random walk, that is, when $\tS_t = \sum_{k=1}^t \eta_k$ and $(\eta_k)_{k\ge1}$ are i.i.d.\ uniform on $\{-1,1\}$, the classic computation 
	\begin{equation*}
	F(u) = \left( \frac{1}{1-u} - \frac{1-\sqrt{1-u^{2}}}{u(1-u)} \right)\sim \sqrt{\frac{2}{1-u} } \quad \text{as} \quad u\rightarrow 1^{-}
	\end{equation*}  
	gives 
	\begin{equation}\label{eq:SRW_asympotitc}
	\mathbb{P}(T_{0}\ge t)  \sim \sqrt{2/\pi t}.
	\end{equation}
	
	Many classical results of random walks with i.i.d.~increments has been generalized to broader classes of random walks with correlations between increments. One such class of processes are called Markov additive processes (MAPs), which is a two dimensional Markov process $(S_{t}, E_{t})$, where $S_t$ is a real-valued additive process and $E_t$ is an underlying state variable. Various aspects of MAPs have been studied, including analogues of Spitzer's identity and Wiener-Hopf factorization. See, for example, Presman \cite{presman1969factorization} and Arjas and Speed \cite{arjas1973symmetric} for discrete time MAPs, and Klusik and Palmowski \cite{klusik2014note} and Ivanovs \cite{ivanovs2017splitting} for more recent accounts of continuous time MAPs. To our knowledge, a precise analogue of (\ref{eq:SRW_asympotitc}) for MAPs has not been derived explicitly in the literature.
	
	In the present paper, we obtain an analogue of (\ref{eq:SRW_asympotitc}) for a special case of discrete time MAPs, namely the \textit{additive functionals of Markov chains} \cite{blumenthal1964additive}, where the increments $S_{n+1}-S_{n}$ are given by a functional of the underlying state, $g(E_{n})$. Our approach relies on a simple combinatorial observation, given in Lemma~\ref{lemma:key}, which allows us to relate the survival probability to the running maximum of a time-reversed random walk, which is a well-understood process. This relationship (Theorem \ref{thm:key_relation}) holds for the general case of stationary processes. To obtain the desired asymptotic in Theorem \ref{thm:general_survival_Brownian}, a functional central limit theorem for additive functionals of Markov chains is used to extract the asymptotic of the expected running maximum of a time-reversed process.

	We now give precise statements of our main results. Consider a bi-infinite discrete-time stationary process $(\mathtt{X}_{t})_{t\in \mathbb{Z}}$ on general state space $(\mathfrak{X},\mathcal{F})$. Let $\PP$ denote the law of the process and $\EE$ denote expectation with respect to $\PP$.  Let $g:\mathfrak{X}\rightarrow \mathbb{R}$ be a measurable function with $\mathbb{E}(g(\mathtt{X}_{0}))=0$. Define the random walk $(\mathtt{S}_{t})_{t\ge 0}$ by $\mathtt{S}_{t}=\mathtt{S}_{t-1}+g(\mathtt{X}_{t})$ for $t\ge 1$ with initial condition $\mathtt{S}_{0}=r$. For each $r\in \mathbb{R}$, $x\in \mathfrak{X}$, and $t\ge 0$, define \textit{survival probabilities} $\mathtt{Q}^{x}(r,t)$ and $\mathtt{Q}^{\bullet}(r,t)$ by  
	\begin{equation}\label{def:conditional_survival}
	\begin{aligned}
	\mathtt{Q}^{x}(r,t) &:= \mathbb{P}(\mathtt{S}_{1}\ge 0,\cdots, \mathtt{S}_{t}\ge 0\,|\, \mathtt{S}_{0}=r,\, \mathtt{X}_{0}=x),\quad \text{and} \\ \mathtt{Q}^{\bullet}(r,t) &:= \mathbb{P}(\mathtt{S}_{1}\ge 0,\cdots, \mathtt{S}_{t}\ge 0\,|\, \mathtt{S}_{0}=r).
	\end{aligned}
	\end{equation}
	Define the \textit{backward random walk} $(\cev\tS_{t})_{t\ge 0}$ by $\cev\tS_t =\sum_{s=-t}^{-1} g(\mathtt{X}_{s})$ for $t\ge 1$ and $\cev\tS_0=0$, and the \textit{backward running maximum} $\cev{M}(t)$ by
	\begin{equation}\label{def:backward_maximum}
	\cev{M}(t) = \max_{0\le k \le t} \cev\tS_{k}
	\end{equation}
for $t\ge 0$.

	Our key observation is the following general relation between the asymptotic behavior of the expected backward running maximum and the survival probability integrated over starting position, $r$.  
	
	\begin{customthm}{1}\label{thm:key_relation}
		Let $\mathtt{Q}^{\bullet}(r,t)$ and $\cev{M}(t)$ be defined as in~\eqref{def:conditional_survival} and~\eqref{def:backward_maximum}. For any constants $C>0$ and $\rho\in (0,1)$,  the following two statements are equivalent:
		\begin{description}
			\item[(i)] $\mathbb{E}(\cev{M}(t))\sim  C t^{1-\rho}$, 
			\vspace{0.1cm}
			\item[(ii)] $\int_{0}^{\infty} \mathtt{Q}^{\bullet}(-r,t) \,dr \sim \rho C t^{-\rho}.$
		\end{description}
	\end{customthm}
	Theorem~\ref{thm:key_relation} is useful because in many cases the asymptotics of $\mathbb{E}(\cev{M}(t))$ are easily accessible via a functional central limit theorem, so we can obtain the precise asymptotics for the integrated survival probability. 

	Now we take $(\mathtt{X}_{t})_{t\in \mathbb{Z}}$ to be a strictly stationary Markov chain on $\mathfrak{X}$ with transition kernel $P$ and unique stationary measure $\pi$. 
	Suppose $\EE(g(\mathtt{X}_{0})^2)<\infty$ and the following quantity   
	\begin{equation}\label{def:limiting_variance}
	\gamma_{g}^{2} := \text{Var}[g(\mathtt{X}_{0})]+2\sum_{k=1}^{\infty}\text{Cov}[g(\mathtt{X}_{0}),g(\mathtt{X}_{k})]
	\end{equation}
	is finite. Here $\gamma^{2}_{g}$ is called the $\textit{limiting variance}$ of the additive processes $\mathtt{S}_{n}$ and also for $\cev{\mathtt{S}}_{n}$.
	Further assume that the time-reversed chain $(\mathtt{X}_{-t})_{t\ge 0}$ is ergodic. Then the scaling limit of the backward running maximum $\cev{M}(t)$ is readily accessible by a functional central limit theorem for Markov chains. Namely, if we let $[\cev\tS](\cdot):[0,\infty)\rightarrow \mathbb{R}$ denote the linear interpolation of the points $((t,\cev\tS_{t}))_{t\ge 0}$, then 
\begin{equation*}
\left( t^{-1/2}[\cev\tS](ts)\,:\,0\le s \le 1\right) \overset{d}{\longrightarrow} \gamma_{g} B
\end{equation*}
as $t\rightarrow \infty$, where $B=(B_{s}\,:\, 0\le s\le 1)$ is the standard Brownian motion and $\gamma_{g}=\sqrt{\gamma_{g}^{2}}$. See, for example, Corollary 3 of Dedecker and Rio \cite{dedecker2000functional} or Theorem 17.4.4 of Meyn and Tweedie \cite{meyn2012markov}. From this and uniform integrability, which is proved in \cite{dedecker2000functional}, we get
\begin{equation}\label{eq:maximum_asymptotic}
\mathbb{E}(\cev{M}(t)) \sim \gamma_{g} \sqrt{\frac{2t}{\pi}}.
\end{equation}
This, combined with Theorem~\ref{thm:key_relation}, gives the following result.

\begin{customthm}{2}\label{thm:general_survival_Brownian}
	Suppose $\mathbb{E}(g(\mathtt{X}_{0})^{2})<\infty$, $\gamma_{g}\in (0,\infty)$, and $(X_{-t})_{t\ge 0}$ is ergodic. Then we have 
	\begin{equation}
	\int_{0}^{\infty} \mathtt{Q}^{\bullet}(-r,t) \,dr \sim \frac{\gamma_{g}}{\sqrt{2\pi}} t^{-1/2}.
	\end{equation}
\end{customthm}

	We emphasize that Theorem \ref{thm:general_survival_Brownian} applies even when the increments have long-range correlation, as opposed to the classical Sparre-Anderson setting with independent increments. 
	
	As an application of Theorem \ref{thm:general_survival_Brownian}, we establish a clustering behavior of the 3-color \textit{firefly cellular automaton} (FCA) on the one dimensional integer lattice $\mathbb{Z}$. The FCA was introduced recently by the first author as a discrete model for pulse-coupled oscillators \cite{lyu2015synchronization}. Namely, let $G=(V,E)$ be a simple graph and an initial 3-coloring $X_{0}:V\rightarrow \mathbb{Z}_{3}$ is given. The 3-color FCA trajectory $(X_{t})_{t\ge 0}$ on $G$ is a discrete-time dynamical process generated by iterating the following transition map
	\begin{equation}\label{FCA_transition_rule}
	(\text{FCA})\qquad X_{t+1}(v)=\begin{cases}
	X_t(v) & \text{if $X_t(v) =2$ and $\exists u\in N(v)$ s.t. $X_{t}(u)=1$}\\
	X_t(v)+1 \text{ (mod $3$)} & \text{otherwise,}\end{cases}
	\end{equation}
	where $N(v)$ denotes the set of all neighbors of $v$ in $G$. (See Section \ref{section:application} for backgrounds and details.) We show that if the discrete oscillators form an infinite array whose colors are initialized at random, then they tend to synchronize locally. More precisely, we obtain the following asymptotic of the probability of local disagreement. 
	
	\begin{customthm}{3}\label{thm:3FCA_Z}
		Let $G=\mathbb{Z}$ be the one-dimensional integer lattice with nearest-neighbor edges. Let $(X_{t})_{t\ge 0}$ be the 3-color FCA trajectory on $\mathbb{Z}$, where $X_{0}$ is a random 3-coloring drawn from the uniform product measure on $(\mathbb{Z}_{3})^{\mathbb{Z}}$. Then for any $x\in \mathbb{Z}$, as $t\rightarrow \infty$, we have 
		\begin{equation}
			\mathbb{P}(X_{t}(x)\ne X_{t}(x+1)) \sim \sqrt{\frac{8}{81}}t^{-1/2}.
		\end{equation}
	\end{customthm}

	  This paper is organized as follows. In Section \ref{section:proof} we prove Theorems \ref{thm:key_relation} and \ref{thm:general_survival_Brownian}. For the special case when the state space $\mathfrak{X}$ is finite, we give an explicit formula for the limiting variance $\gamma_{g}^{2}$ in Proposition \ref{prop:gamma_formula}. Furthermore, when the functional $g$ is integer-valued, we give asymptotic relations of various survival probabilities $\mathtt{Q}^{x}(j,t)$, which allows us to extract their asymptotics from Theorem \ref{thm:general_survival_Brownian}. In Section \ref{section:application}, we apply our results to study asymptotic behavior of three 3-color cellular automata models of excitable media on $\mathbb{Z}$, namely, the cyclic cellular automaton (CCA), the Greenberg-Hastings model (GHM), and the FCA. We readily recover known results for CCA and GHM, and prove Theorem \ref{thm:3FCA_Z} for the FCA.

	\vspace{0.5cm}
	\section{Proof of main results}
	\label{section:proof}
	
	In this section we prove the main results, Theorems \ref{thm:key_relation} and \ref{thm:general_survival_Brownian}. The key thread connecting the two different quantities in Theorem \ref{thm:key_relation} is provided in the following simple lemma. 
	
	\begin{customlemma}{2.1}\label{lemma:key}
		Let $\mathtt{Q}^{\bullet}(r,t)$ and $\cev{M}(t)$ be defined as in~\eqref{def:conditional_survival} and~\eqref{def:backward_maximum}. Then for all $r>0$ and $t\in \mathbb{N}$,
		\begin{equation}\label{eq:lemma_key}
		\mathbb{P}(\cev{M}(t) - \cev{M}(t-1)\ge r ) =  \mathtt{Q}^{\bullet}(-r,t-1).
		\end{equation} 
	\end{customlemma}
	
	\begin{proof}
		Fix $t\ge 1$, and let $\mathtt{S}_{k}=\mathtt{S}_0+\sum_{j=1}^{k} g(\mathtt{X}_{j})$ for all $1\le k \le t$. Consider the time-reversed sequence  $(g(\mathtt{X}_{t}),g(\mathtt{X}_{t-1}),\cdots,g(\mathtt{X}_{1}))$. For $1\le k \le t$, denote by $M_{k}$ the maximum of the first $k$ partial sums of this sequence,
		\begin{equation*}
		\begin{aligned}
		M_{k} &= \max_{1\le j \le k} \left( g(\mathtt{X}_{t})+ \cdots + g(\mathtt{X}_{t-j+1})\right)\\
		&=\max_{1\le j \le k} \left( \mathtt{S}_{t} - \mathtt{S}_{t-j}\right) \\
		 &= \mathtt{S}_{t} - \min_{1\le j \le k} \mathtt{S}_{t-j},
		\end{aligned}
		\end{equation*}
		and note that $M_k$ does not depend on $\mathtt{S}_0$ for any $1\le k\le t$. Then since $r>0$, observe 
		\begin{eqnarray*}
			\left\{ M_{t}-M_{t-1}\ge r \right\} &=& \left\{ \min\{ \mathtt{S}_{1},\mathtt{S}_{2},\cdots,\mathtt{S}_{t-1} \} - \min\{ \mathtt{S}_0,\mathtt{S}_{1},\mathtt{S}_{2},\cdots,\mathtt{S}_{t-1} \} \ge r \right\} \\
			&=& \left\{ \mathtt{S}_{1} \ge r+\mathtt{S}_0,\, \mathtt{S}_{2}\ge r+\mathtt{S}_0,\, \cdots, \mathtt{S}_{t-1} \ge r+\mathtt{S}_0 \right\}.
		\end{eqnarray*}
		If we assume $\mathtt{S}_0=-r$, then we have 
		\begin{equation*}
			\mathbb{P}(  M_{t}-M_{t-1}\ge r ) = \mathtt{Q}^{\bullet}(-r,t-1).
		\end{equation*}
		To finish the proof, note that by the stationarity of $(\mathtt{X}_{t})_{t\in \mathbb{Z}}$, for each fixed $t\ge 1$, the random vectors $(\mathtt{X}_{t},\mathtt{X}_{t-1},\cdots,\mathtt{X}_{1})$ and $(\mathtt{X}_{-1},\mathtt{X}_{-2},\cdots,\mathtt{X}_{-t})$ have the same joint distribution. Thus so do $(M_{1},M_{2},\cdots,M_{t})$ and $(\cev{M}(1),\cev{M}(2),\cdots,\cev{M}(t))$. Hence the assertion follows. 
	\end{proof}
	
	The proof of the above lemma was inspired by analyzing a class of one-dimensional cellular automata and their embedded particle systems in two different perspectives. We discuss more in this direction in Section \ref{section:application}.

	\vspace{0.2cm}
	In what follows, we will make use of a Tauberian theorem for power series, which we state in Theorem \ref{thm:tauberian} below. 
	
	\begin{customthm}{2.2}[Theorem XIII.5.5 in \cite{feller1971introduction} ]\label{thm:tauberian}
		Let $(q_{n})_{n\ge 0}$ be a sequence of positive reals and suppose that 
		\begin{equation*}
		Q(u):=\sum_{n=0}^{\infty} q_{n}u^{n}
		\end{equation*}
		converges for all $u\in [0,1)$. If $L$ varies slowly at infinity and $0\le \rho<\infty$, then each of the two relations 
		\begin{equation*}
		Q(u)\sim (1-u)^{-\rho} L((1-u)^{-1}) \quad \text{as} \quad u\rightarrow 1^{-}
		\end{equation*}
		and 
		\begin{equation*}
		\sum_{k=0}^{n-1}q_{k} \sim \frac{n^{\rho}}{\Gamma(\rho+1)} L(n) \quad \text{as} \quad n\rightarrow \infty
		\end{equation*}
		are equivalent. Furthermore, if $(q_{n})_{n\ge 0}$ is monotone and $0<\rho<\infty$, then the first relation is equivalent to 
		\begin{equation*}
		q_{n} \sim \frac{n^{\rho-1}}{\Gamma(\rho)} L(n) \quad \text{as} \quad n\rightarrow \infty.
		\end{equation*}
	\end{customthm}

	Now we are ready to give a proof of Theorem \ref{thm:key_relation}. 
	
	\vspace{0.1cm}
	\hspace{-0.42cm}\textbf{Proof of Theorem \ref{thm:key_relation}.} Define $\mathtt{e}(t) = \cev{M}(t+1)-\cev{M}(t)$ for $t\ge 0$. Since $\cev{M}(t)$ is non-deceasing in $t$, $\mathtt{e}(t)\ge 0$ for all $t\ge 0$. By integrating both sides of (\ref{eq:lemma_key}) over $r\in (0,\infty)$ with respect to the Lebesgue measure, and using the fact that both integrands are bounded by $1$ at $r=0$, we get 
	\begin{equation*}
		\mathbb{E}(\mathtt{e}(t)) =  \int_{0}^{\infty} \mathbb{P}(\mathtt{e}(t) \ge r) \,dr =\int_{0}^{\infty}\mathtt{Q}^{\bullet}(-r,t)\,dr.  
	\end{equation*}
	Hence it suffices to show the assertion with the integrated survival probability replaced by $\mathbb{E}(\mathtt{e}(t))$. Moreover, note that the above equation implies that $\mathbb{E}(\mathtt{e}(t))$ is non-increasing in $t$, since the integrand $\mathtt{Q}^{\bullet}(-r,t)$ is non-increasing in $t$. 
	
	Now define $G(u)$ and $H(u)$ to be the generating functions for $\mathbb{E}(\mathtt{e}(t))$ and $\mathbb{E}(\cev{M}(t))$ (noting that $\cev{M}(0)=0$):
	\begin{equation}
	G(u) = \sum_{t=0}^{\infty} \mathbb{E}(\mathtt{e}(t)) u^{t}, \qquad H(u) = \sum_{t=1}^{\infty} \mathbb{E}(\cev{M}(t)) u^{t},
	\end{equation}
	where $u\in [0,1)$. By linearity of expectation we have $\mathbb{E}(\cev{M}(t))=\sum_{k=0}^{t-1} \mathbb{E}(\mathtt{e}(k))$, so  
	\begin{equation*}
	H(u) = \sum_{t=1}^{\infty} \sum_{k=0}^{t-1} \mathbb{E}(\mathtt{e}(k))u^{t} = \sum_{k=0}^{\infty} \mathbb{E}(\mathtt{e}(k)) \sum_{t=k+1}^{\infty} u^{t} = u(1-u)^{-1}\sum_{k=0}^{\infty} \EE (\mathtt{e}(k)) u^{k} = u(1-u)^{-1}G(u).
	\end{equation*}
	 Hence it follows that 
	\begin{equation*}
		\lim_{u\rightarrow 1^{-}} (1-u)^{\rho} G(u) = \lim_{u\rightarrow 1^{-}} (1-u)^{1+\rho} H(u),
	\end{equation*}
	whenever one of the two limits exists.
	The assertion follows from the monotonicity of $\mathbb{E}(\cev{M}(t))$ and $\mathbb{E}(\mathtt{e}(t))$, Theorem~\ref{thm:tauberian}, and the fact that $\Gamma(1+\rho)/\Gamma(\rho)=\rho$. $\blacksquare$

	\vspace{0.2cm}
	Now in order to prove Theroem \ref{thm:general_survival_Brownian}, it suffices to compute the desired asymptotic of $\mathbb{E}(\cev{M}(t))$. This can be achieved easily by a functional central limit theorem for Markov chains.
	
	\vspace{0.1cm}
	\hspace{-0.42cm}\textbf{Proof of Theorem \ref{thm:general_survival_Brownian}.} By assumption, the backward chain $(\mathtt{X}_{-t})_{t\ge 0}$ is strictly stationary and ergodic.
	Moreover, the associated partial sums process $\cev{\mathtt{S}}_{n}$ also has the same limiting variance $\gamma_{g}^{2}$ as the forward process, which is assumed to be finite. Hence by a functional CLT for Markov chains (Corollary 3 of \cite{dedecker2000functional}), we have 
	\begin{equation}
	\frac{\cev{M}(t)}{\gamma_{g}\sqrt{t}} \overset{d}{\longrightarrow} \max_{0\le s \le 1} B_{s} \quad \text{as $t\rightarrow \infty$.}
	\end{equation}
	Moreover, Proposition 1 of \cite{dedecker2000functional} implies that the sequence $(t^{-1/2}\cev{M}(t))_{t\ge 1}$ is uniformly integrable. Hence we get  
	\begin{equation}
	\lim_{t\rightarrow \infty} \frac{\mathbb{E}(\cev{M}(t))}{\gamma_{g}\sqrt{t}} = \int_{0}^{\infty}\mathbb{P}(|Z|>r)\,dr = \sqrt{\frac{2}{\pi}},
	\end{equation}
	where $Z\sim N(0,1)$ is the standard normal random variable. This shows the desired asymptotic for $\mathbb{E}(\cev{M}(t))$.  $\blacksquare$

	\begin{customremark}{2.3}(Skip-free walks) An integer-valued random walk $\mathtt{S}_{n}$ is called \textit{upward} (resp., \textit{downward}) \textit{skip-free} if the only positive (resp., negative) value that its increment can take is $1$ (resp., $-1$). Downward skip-free walks with i.i.d.\ increments naturally arise in the context of M/G/1 queue at departure times of customers (see \cite{latouche1999introduction}, p. 267), and also in the breadth-first-search process of branching processes \cite{van2009random}. The most relevant discussion about survival probabilities of (downward) skip-free random walks with i.i.d. increments can be found in Aurzada and Simon \cite{aurzada2015persistence}. The discussion is based on the following so-called random walk hitting time theorem (see e.g., \cite{van2008elementary}). If the increments take values from $\{-1,0,1,2,\cdots\}$, then this theorem states that 
		\begin{equation}\label{eq:hittingtimethm}
		\mathtt{Q}^{\bullet}(k,t) = \frac{k+1}{t} \mathbb{P}(\mathtt{S}_{t}=k+1).
		\end{equation}
	Then a local limit theorem gives the asymptotic of $\mathbb{P}(\mathtt{S}_{t}=k+1)$, and hence that of $\mathtt{Q}^{\bullet}(k,t)$. On the other hand, our Theorem \ref{thm:general_survival_Brownian} is well-suited to handle the case of upward skip-free walks, with increments taking values in $\{\cdots, -2, -1, 0, 1\}$. In this case, we have
	\begin{equation*}
		\int_{0}^{\infty} \mathtt{Q}^{\bullet}(-r,t) \,dr = \mathtt{Q}^{\bullet}(-1,t),
	\end{equation*}
	since the only positive step is $1$, and all steps are integer valued. Therefore, Theorem \ref{thm:general_survival_Brownian} gives the precise asysmptotic of $\mathtt{Q}^{\bullet}(-1,t)$ for upward skip-free walks.
	\end{customremark}

	
	\vspace{0.5cm}
	\section{Computation and examples}

	Note that Theorem \ref{thm:general_survival_Brownian} applies whenever the state space $\mathfrak{X}$ is finite and the Markov chain $(\mathtt{X}_t)_{t\in \Z}$ is stationary and ergodic. We assume this throughout this section and provide an explicit formula for the limiting variance $\gamma_{g}^{2}$ as well as some examples. Before we get to the general case, we first give some preliminary examples to illustrate computation for the general case.

	\begin{customexample}{3.1}(Persistent random walks.)
		Take $\mathfrak{X}=\{-1,1\}$ and let $(\mathtt{X}_{n})_{n\ge 0}$ be a Markov chain on $\mathfrak{X}$ with transition matrix $P=(a_{ij})$ and stationary measure $\pi$. Let $g:\mathfrak{X}\rightarrow \mathbb{Z}$ be the identity map. Since we require $g(\mathtt{X}_{0})$ has mean zero under the stationary measure, we need to have $\pi=(1/2,1/2)$. This implies $P$ is of the form 
		\begin{equation*}
		P=\begin{bmatrix}
		\alpha & \beta\\
		\beta & \alpha
		\end{bmatrix}
		=
		\begin{bmatrix}
		1 & 1\\
		0 & -1
		\end{bmatrix}
		\begin{bmatrix}
			1 & 0\\
			0 & \alpha-\beta
		\end{bmatrix}
		\frac{1}{2}
		\begin{bmatrix}
		1 & 1\\
		0 & -1
		\end{bmatrix}
		\end{equation*} 
		for some $0\le \alpha,\beta\le 1$ with $\alpha+\beta=1$. The resulting random walk $(\mathtt{S}_{n})_{n\ge 0}$ is called a \textit{persistent} (resp., \textit{anti-persistent}) random walk if $\alpha>\beta$ (resp. $\alpha<\beta$). In words, its next increments keep the same sign with probability $\alpha$ and flips with probability $\beta$. This process is related to the telegrapher's equation and is first introduced by Kac \cite{kac1974stochastic}. A special case is when $\alpha=\beta$, in which case $(\mathtt{S}_{n})_{n\ge 0}$ is the simple symmetric random walk. 
		
		A direct computation shows  
		\begin{equation*}
		\text{Cov}(\mathtt{X}_{0}\mathtt{X}_{k})=\mathbb{E}(\mathtt{X}_{0}\mathtt{X}_{k}) = (\alpha-\beta)^{k} \quad \forall k\ge 0,
		\end{equation*}
		so we have 
		\begin{equation*}
		\gamma_{g}^{2} = 1+ 2 \frac{\alpha-\beta}{1-(\alpha-\beta)} = \alpha/\beta.
		\end{equation*}
		Hence Theorem \ref{thm:general_survival_Brownian} gives 
		\begin{equation*}
		\mathtt{Q}^{\bullet}(-1,t)\sim \sqrt{\frac{\alpha}{2\pi \beta}} t^{-1/2}.
		\end{equation*}
		Noting that $\mathtt{Q}^{\bullet}(-1,t)=(1/2)\mathtt{Q}^{1}(0,t-1)$, this yields 
		\begin{equation*}
		\mathtt{Q}^{1}(0,t)\sim \sqrt{\frac{2\alpha}{\pi \beta}} t^{-1/2},
		\end{equation*}
		which agrees with (\ref{eq:SRW_asympotitc}) in the special case of $\alpha=\beta$. $\blacktriangle$
	\end{customexample}

	\begin{customexample}{3.2}(Markovian increments on $\mathfrak{X}=\{-1,0,1\}$.) Let $\mathfrak{X}=\{-1,0,1\}$ and let $(\mathtt{X}_{n})_{n\ge 0}$ be a Markov chain on $\mathfrak{X}$ with transition matrix $P=(a_{ij})$ and stationary measure $\pi$. Let $g:\mathfrak{X}\rightarrow \mathbb{Z}$ be the identity map. It is easy to see that $\pi=(a,b,a)$ for some $0\le a, b \le 1$ (or equivalently $\mathbb{E}(g(\mathtt{X}_{0}))=0$) if and only if 
		\begin{equation}
		\frac{b}{a}=\frac{1-a_{11}-a_{31}}{a_{21}} = \frac{1-a_{13}-a_{33}}{a_{23}} = \frac{a_{21}+a_{23}}{1-a_{22}}.
		\end{equation}
		In order to compute the limiting variance $\gamma_{g}^{2}$, observe that 
		\begin{eqnarray*}
			\mathbb{E}(\mathtt{X}_{0}\mathtt{X}_{k}) &=& a( [P^{k}]_{11}+[P^{k}]_{33}-[P^{k}]_{13}-[P^{k}]_{31}) \\
			&=& a  \begin{bmatrix}
				1 & 0 & -1
			\end{bmatrix} 
			P^{k}
			\begin{bmatrix}
				1 & 0 & -1
			\end{bmatrix}^{T}
		\end{eqnarray*}
		for all $k\ge 1$. To simplify the expression, we write $P=UAU^{-1}$, where $A$ is the Jordan normal form of $P$ and $U$ is the invertible matrix of a Jordan basis. There are following two possibilities
		\begin{equation}
		P = U 
		\begin{bmatrix}
		1 & 0 & 0\\
		0 & \lambda_{2} & 0 \\
		0 & 0 & \lambda_{3} 
		\end{bmatrix}
		U^{-1}
		\quad \text{or} \quad 
		U 
		\begin{bmatrix}
		1 & 0 & 0\\
		0 & \lambda_{2} & 1 \\
		0 & 0 & \lambda_{2} 
		\end{bmatrix}
		U^{-1}.
		\end{equation} 	
		Interesting cases are when $|\lambda_{2}|,|\lambda_{3}|<1$, which we may assume in this example. Write $U=(v_{ij})$ and $U^{-1}=(u_{ij})$. Note that $(v_{11},v_{21},v_{31})=(1,1,1)$. So in the first case when $A$ is a diagonal matrix,  we get 
		\begin{eqnarray*}
			\mathbb{E}(\mathtt{X}_{0}\mathtt{X}_{k}) =  a\sum_{j=2}^{3} \lambda_{j}^{k}(v_{j1}-v_{j3})(u_{j1}-u_{j3})
		\end{eqnarray*}
		for all $k\ge 1$. So we obtain 
		\begin{equation*}
		\gamma_{g}^{2} = 2a + 2a\sum_{j=2}^{3} \frac{\lambda_{j}}{1-\lambda_{j}}(v_{j1}-v_{3j})(u_{j1}-u_{3j}).
		\end{equation*}
		
		In the latter case, similar computation shows 
		\begin{eqnarray*}
			\mathbb{E}(\mathtt{X}_{0}\mathtt{X}_{k}) = a\left( \sum_{j=2}^{3} \lambda_{2}^{k}(v_{j1}-v_{3j})(u_{j1}-u_{3j}) \right)+ ak\lambda_{2}^{k-1}(v_{21}-v_{32})(u_{31}-u_{33}),
		\end{eqnarray*}
		so 
		\begin{equation*}
		\gamma_{g}^{2} = 2a + 2a\sum_{j=2}^{3} \frac{\lambda_{2}}{1-\lambda_{2}}(v_{j1}-v_{3j})(u_{j1}-u_{3j})+a(1-\lambda_{2})^{-2}(v_{21}-v_{32})(u_{31}-u_{33}).
		\end{equation*}

		For a concrete example, consider 
		\begin{equation*}
		P = \begin{bmatrix}
		2/3 & 0 & 1/3 \\
		1/6 & 1/2 & 1/3 \\
		1/4 & 1/4 & 1/2 
		\end{bmatrix}
		=
		\begin{bmatrix}
		1 & -2 & -2/3 \\
		1 & 2 & -2/3 \\
		1 & 1 & 1 
		\end{bmatrix}
		\begin{bmatrix}
		1 & 0 & 0 \\
		0 & 1/2 & 0 \\
		0 & 0 & 1/6 
		\end{bmatrix}
		\begin{bmatrix}
		1 & -2 & -2/3 \\
		1 & 2 & -2/3 \\
		1 & 1 & 1 
		\end{bmatrix}^{-1}.
		\end{equation*}
		Then $\pi=(1/3,1/3,1/3)$, and the above formula gives  
		\begin{equation*}
		\mathbb{E}(\mathtt{X}_{0}\mathtt{X}_{k}) = \frac{1}{3} \left(  \frac{1}{2^{k}} + \frac{2}{5 \cdot 6^{k}}  \right) \quad \forall k\ge 1,
		\end{equation*}
		so we obtain 
		\begin{equation*}
		\gamma_{g}^{2} = 101/75.
		\end{equation*}
		Thus Theorem \ref{thm:general_survival_Brownian} gives 
		\begin{equation*}
		\mathtt{Q}^{1}(-1,t)\sim \sqrt{\frac{101}{150 \pi }} t^{-1/2}.
		\end{equation*}
		$\blacktriangle$
	\end{customexample}
	
	Now we compute the limiting variance $\gamma_{g}^{2}$ in the general case when the state space $\mathfrak{X}$ is finite. Let $\mathfrak{X}=\{x_{1},\cdots,x_{n}\}$ for some $n\in \mathbb{N}$, and let $P$ be the transition matrix with stationary distribution $\pi=(p_{1},\cdots,p_{n})$. We may write $P=UAU^{-1}$ where $A$ is a Jordan normal form of $P$ and $U=(u_{ij})$ is an invertible matrix of a Jordan basis. Let $m$ be the number of distinct Jordan blocks of $A$, and for each $1\le i \le m$, write its $i^{\text{th}}$ Jordan block $J^{ii}$ by $J^{ii}=\lambda_{i}I_{i}+N_{i}$, where $I_{i}$ is the identity matrix of dimension $m_{i}\ge 1$ and $N_{i}$ is a square matrix of the same dimension whose entries are all zeros but 1's just above the diagonal. We write $A$ into a block form $A=(J^{ij})$ for $1\le i,j\le m$, where $J^{ij}=O$ with the appropriate dimension whenever $i\ne j$. With the same partitioning into blocks, we may write $U=(U^{ij})$ and $V=(V^{ij})$. Let $G=[g(x_{1}) \cdots g(x_{n})]^{T}$ and $H=[p_{1}g(x_{1}) \cdots p_{n}g(x_{n})]$, and write these matrices into the corresponding block form $G=(G^{j1})$ and $H=(H^{1i})$. Finally, for each $1\le i \le m$ with $|\lambda_{i}|<1$, define 
	\begin{equation}
	\bar{J}^{ii} = \left( \lambda_{i}(1-\lambda_{i})^{-1}I_{i}+ \sum_{\ell=1}^{m_{i}-1} (-1)^{\ell} (1-\lambda_{i})^{-\ell-1} N_{i}^{\ell} \right).
	\end{equation}

	\begin{customprop}{3.3}\label{prop:gamma_formula}
		With the notation given above, we have 
		\begin{equation}\label{eq:gamma_g_nMC_general}
		\gamma_{g}^{2}
		= \sum_{i=1}^{n} p_{i}g(x_{i})^{2}+ 2\sum_{\substack{1\le i,k\le m \\ |\lambda_{j}|< 1}} H^{1i} U^{ij}\bar{J}^{jj} V^{jk} G^{k1}.
		\end{equation}
		In particular, if $A=\text{diag}(1,\lambda_{2},\cdots,\lambda_{n})$, then
		\begin{equation}\label{eq:gamma_g_nMC}
		\gamma_{g}^{2}
		= \sum_{i=1}^{n} p_{i}g(x_{i})^{2} + 2 \sum_{\substack{1\le j \le n\\ |\lambda_{j}|< 1}} \frac{\lambda_{j}}{1-\lambda_{j}} \sum_{1\le i,k\le n} p_{j}g(x_{i})g(x_{k})u_{ij}v_{jk}.
		\end{equation}
	\end{customprop}
	
	\begin{proof}
		The second part of the assertion follows immediately from the first part. To show the first pat, we begin by noting that for each $k\ge 1$, 
		\begin{eqnarray*}
			\mathbb{E}(g(\mathtt{X}_{0})g(\mathtt{X}_{k})) &=& \sum_{1\le i,j\le n} p_{i}g(x_{i}) [P^{k}]_{ij} g(x_{j})   \\
			&=& H U A^{k} U^{-1} G \\
			&=& \sum_{\substack{1\le p,q,r\le n\\ |\lambda_{q}|<1}}  H^{1p} U^{pq} (\lambda_{q}I_{q}+N_{q})^{k} V^{qr} G^{r1}\\
			 && \qquad \qquad + \sum_{\substack{1\le p,q,r\le n\\ |\lambda_{q}|=1}}  H^{1p} U^{pq} (\lambda_{q}I_{q}+N_{q})^{k} V^{qr} G^{r1}\\
		\end{eqnarray*}
	For $1\le q \le m$ such that $|\lambda_{q}|<1$, we have 
	\begin{equation*}
	\sum_{k\ge 1}   (\lambda_{q}I_{q}+N_{q})^{k} = \sum_{k\ge 1}   \lambda_{q}^{k}I_{q}+\binom{k}{1}\lambda_{q}^{k-1} N_{q} + \cdots + \binom{k}{m_{q}-1}\lambda_{q}^{k-m_{q}+1} N_{q}^{m_{q}-1},
	\end{equation*}
	where we take $\binom{a}{b}=0$ whenever $a<b$. It is then easy to see that the above geometric series of matrices converges to $\bar{J}^{qq}$. Hence by linearity it suffices to show that the contribution of the second summation above to $\gamma_{g}^{2}$ is zero. 
		
		To this end, we first note that any eigenvalue $\lambda_{j}$ of $P$ of modulus 1 has equal algebraic and geometric multiplicity \cite{ding2011equality}, so the corresponding Jordan block $J_{j}$ is of size 1. For each $1\le j \le n$ such that $\lambda_{j}=1$, since $u_{ij}=u_{kj}$ for all $1\le i,j\le n$ and $\mathbb{E}(g(\mathtt{X}_{0}))=p_{1}g(x_{1})+\cdots+p_{n}g(x_{n})=0$, there is no contribution to the above summation whenever $\lambda_{j}=1$. Hence if we denote by $\{\mu_{1},\cdots,\mu_{n^{*}}\}$ the set of all distinct eigenvalues of $P$ of modulus 1 but not equal to 1,  
		\begin{equation*}
		\sum_{k=1}^{N} \,	\sum_{\substack{1\le p ,q,r\le n\\ |\lambda_{q}|=1}}  H^{1p} U^{pq} (\lambda_{q}I_{q}+N_{q})^{k} V^{qr} G^{r1} = \sum_{1\le j \le n^{*}}C(\mu_{i})\frac{1-\mu_{i}^{N}}{1-\mu_{r}}
		\end{equation*}
		where $C(\mu_{i})$ is a constant that does not depend on $N$. Since this summation must converge as $N\rightarrow \infty$, we must have that $C(\mu_{i})=0$ for all $1\le i \le n^{*}$. This shows the assertion.

	\end{proof}

	\vspace{0.3cm}
	Next, we establish linear relationships between the survival probabilities $\mathtt{Q}^{x}(j,t)$ for different values of $x$ and $j$ in the case when the state space $\mathfrak{X}$ is finite and the functional $g$ assumes integer values. This will allow us to determine the asymptotics of all survival probabilities $\mathtt{Q}^{x}(j,t)$ from Theorem \ref{thm:general_survival_Brownian}. By their monotonicity and Tauberian theorem, we instead work with their generating functions 
	\begin{equation}
	\tilde{\mathtt{Q}}^{x}_{j}(u):=\sum_{t\ge 0} u^{t} \mathtt{Q}^{x}(j,t),\quad \tilde{\mathtt{Q}}^{\bullet}_{j}(u):=\sum_{t\ge 0} u^{t} \mathtt{Q}^{\bullet}(j,t),
	\end{equation}
	where $u\in [0,1)$. Let $P$ be the transition probability matrix for the Markov chain $(\mathtt{X}_t)_{t\ge 0}$ on $\mathfrak{X}$. For each $j\ge i\ge0$ and $x\in \mathfrak{X}$, we define 
	\begin{equation*}
	\tau^{x}_{j}(i) = \inf \left\{ t\ge 0 \,:\, \mathtt{S}_{t}=i, \mathtt{S}_{0}=j, \, \mathtt{X}_{0}=x   \right\}
	\end{equation*}
	to be the first time that $\mathtt{S}_{t}$ hits level $i\le j=\mathtt{S}_{0}$ when $\mathtt{X}_{0}=x$. Also for each $y\in \mathfrak{X}$, let 
	\begin{equation}
	p^{x}_{j,i}(y) = \mathbb{P}\left(  \mathtt{X}_{\tau^{x}_{j}(i)}=y\,\bigg|\, \mathtt{S}_{0}=j,\, \mathtt{X}_{0}=x \right),\quad  p^{x}_{j}(y) = \sum_{i=0}^{j}p^{x}_{j,i}(y),
	\end{equation} 
	and 
	\begin{equation}
	A_{j}(x,y) = \sum_{z\in \mathfrak{X}} \sum_{i=0}^{g(z)} \1\{j+g(z)\ge 0\} P(x,z) p^{z}_{j+g(y),i}(y).
	\end{equation}

	\begin{customprop}{3.4}\label{prop:coefficient}
		For each $x\in \mathfrak{X}$, $j\ge 0$, and $u\in[0,1)$, we have 
		\begin{eqnarray*}
			\tilde{\mathtt{Q}}^{x}_{j}(u)  &\sim &  \sum_{y\in \mathfrak{X}} \1\{ j+g(y)\ge 0 \} P(x,y) \tilde{\mathtt{Q}}^{y}_{j+g(y)}(u), \quad \text{as $u\rightarrow 1^{-}$} \\
			\tilde{\mathtt{Q}}^{x}_{j}(u) &\sim& \sum_{y\in \mathfrak{X}}  p^{x}_{j}(y) \tilde{\mathtt{Q}}_{0}^{y}(u) \quad \text{as $u\rightarrow 1^{-}$}, \\
			\tilde{\mathtt{Q}}^{x}_{0}(u) &\sim & \sum_{y\in \mathfrak{X}} A_{0}(x,y) \tilde{\mathtt{Q}}^{y}_{0}(u) \quad \text{as $u\rightarrow 1^{-}$},
		\end{eqnarray*}
		provided at least one of the quantities in each asymptotic expression diverges to $\infty$ as $u\to~1^-$.
	\end{customprop}
	
	\begin{proof}
		The last asymptotic in the assertion follows by combining the first two. In order to show the first recursion,  we partition on the first step to obtain `forward' recursions. Namely,   
		\begin{equation*}
		\mathtt{Q}^{x}(j,t) = \sum_{y\in \mathfrak{X}} \1\{ j+g(y)\ge 0 \} P(x,y) \mathtt{Q}^{y}(j+g(y),t-1).
		\end{equation*}
		Hence multiplying by $u^{t}$, summing over all $t\ge 1$ and letting $u\to 1^-$ one gets the first recursion in the assertion. 
		
		The second asymptotic in the assertion follows from the `backward' recursions, which can be obtained by partitioning on the value of the minimum of $\mathtt{S}_{s}$ during $s\in [0,t]$,
		\begin{equation*}
		M^{x}_{j}(t):=\min_{0\le k \le t}\{\mathtt{S}_{k}\,|\,\mathtt{S}_{0}=j,\, \mathtt{X}_{0}=x\},
		\end{equation*}
		and on the time of the last visit to this height by time $t$.
		 For each $i \le j$, $0\le s \le t$, and $x,y\in \mathfrak{X}$, let 
		\begin{equation*}
		a^{x}_{j}(i,s,y) = \mathbb{P}\left( \mathtt{S}_{0}> i,\mathtt{S}_{1}> i,\,\cdots, \mathtt{S}_{s-1}> i,\, \mathtt{S}_{s}= i, \, \mathtt{X}_{s}=y \,|\, \mathtt{S}_{0}=j,\, \mathtt{X}_{0}=x \right)
		\end{equation*}	
		be the probability that the random walk hits height $i$ for the first time at time $s$, and at that time the Markov chain is in state $X_s = y$.
		Note that for each $ i \le j$ and $x,y\in \mathfrak{X}$, 
		\begin{eqnarray}\label{eq:coefficient}
		\lim_{u\rightarrow 1^{-}}\sum_{s=0}^{\infty}a^{x}_{j}(i,s,y)u^{s} =\sum_{s=0}^{\infty}a^{x}_{j}(i,s,y)=p^{x}_{j,i}(y). 
		\end{eqnarray}

		By the Markov property of $(\mathtt{X}_{t})_{t\ge 0}$,
		\begin{equation*}
		\begin{aligned}
			\mathtt{Q}^{x}(j,t) &= \mathbb{P}(M^{x}_{j}(t)\ge 0) = \sum_{i=0}^j \mathbb{P}(M^{x}_{j}(t)= i)\\
			&= \sum_{i=0}^j \sum_{s=0}^t \sum_{y\in\mathfrak{X}} \mathbb{P}(M^{x}_{j}(t)= i, \tau^{x}_{j}(i)=s, \mathtt{X}_s = y)\\
			&= \sum_{i=0}^{j}\sum_{s=0}^{t}\sum_{y\in \mathfrak{X}} a^{x}_{j}(i,s,y) \mathtt{Q}^{y}(0,t-s).
		\end{aligned}
		\end{equation*}
		We then multiply by $u^{t}$ and sum over all $t\ge 0$. Using the stationarity of $(\mathtt{X}_{t})_{t\ge 0}$ and Fubini's theorem we obtain 
		\begin{eqnarray*}
			\tilde{\mathtt{Q}}^{x}_{j}(u) 
			&=& \sum_{i=0}^{j}\sum_{t=0}^{\infty}\sum_{s=0}^{t}\sum_{y\in \mathfrak{X}} a^{x}_{j}(i,s,y)u^{t} \mathtt{Q}^{y}(0,t-s)\\
			&=& \sum_{y\in \mathfrak{X}}\sum_{i=0}^{j}\sum_{s=0}^{\infty}  a^{x}_{j}(i,s,y)u^{s}\sum_{t=s}^{\infty} u^{t-s} \mathtt{Q}^{y}(0,t-s)\\
			&=& \sum_{y\in \mathfrak{X}}\sum_{i=0}^{j}\sum_{s=0}^{\infty}  a^{x}_{j}(i,s,y)u^{s}\tilde{\mathtt{Q}}^{y}_{0}(u)\\
			&=& \sum_{y\in \mathfrak{X}}\tilde{\mathtt{Q}}^{y}_{0}(u) \sum_{i=0}^{j}\sum_{s=0}^{\infty}  a^{x}_{j}(i,s,y)u^{s}.
		\end{eqnarray*}
		Hence the second asymptotic in the assertion follows from (\ref{eq:coefficient}).
	\end{proof}

  We illustrate a use of Proposition \ref{prop:coefficient} in the following example.

  \begin{customexample}{3.5}(Non-unit increments)
  	Let $X_{0}:\mathbb{Z}\rightarrow \mathbb{Z}_{3}$ be a random 3-coloring on $\mathbb{Z}$ drawn from the uniform product measure on $(\mathbb{Z}_{3})^{\mathbb{Z}}$. Consider the stationary Markov chain $(\mathtt{X}_{n})_{n\in \mathbb{Z}}$ of color triples, 
  	\begin{equation*}
	  	\mathtt{X}_{n} = (X_{0}(n),X_{0}(n+1),X_{0}(n+2))\in (\mathbb{Z}_{3})^{3}.
  	\end{equation*} 
  	Define a functional $g:(\mathbb{Z}_{3})^{3}\rightarrow [-2,2]$ by 
  	 \begin{eqnarray}
  	 g(i,j,k) = 
  	 \begin{cases}
  	 -2 & \text{if $(i,j,k)=(1,2,0)$}\\
  	 2 & \text{if $(i,j,k)=(0,2,1)$}\\
  	 k-j\in [-1,1] \mod 3& \text{otherwise}. 
  	 \end{cases}
  	 \end{eqnarray}   
  	 Let $(\mathtt{S}_{t})_{t\ge 0}$ be the random walk given by $\mathtt{S}_{t+1}=S_{t}+g(\mathtt{X}_{t})$ and initial condition $\mathtt{X}_{0}=(i,j,k)\in (\mathbb{Z}_{3})^{3}$ and $\mathtt{S}_{0}=\ell\in \mathbb{Z}$. 
  	 
  	 To compute its limiting variance $\gamma_{g}^{2}$, we note that the $27\times 27$ transition matrix $P$ for the chain $(\mathtt{X}_{n})_{n\in \mathbb{Z}}$ is given by 
  	 \begin{equation}
	  	 P((i,j,k),(i',j',k')) = \begin{cases}
		  	 1/3 & \text{if $(j,k)=(i',j')$} \\
		  	 0 & \text{otherwise}
	  	 \end{cases}
  	 \end{equation} 
  	 and its stationary distribution is the uniform probability measure on $(\mathbb{Z}_{3})^{3}$. Hence $g(\mathtt{X}_{0})$ and $g(\mathtt{X}_{k})$ are independent for all $k\ge 3$. A direct calculation yields 
  	 \begin{equation}
	  	 \gamma_{g}^{2} = \frac{2}{729} + 2\left( \frac{2}{729} + \frac{2}{729}  \right) = \frac{10}{729}.
  	 \end{equation}
  	 Hence by Theorem \ref{thm:general_survival_Brownian}, we get 
  	 \begin{equation}
	  	 \mathtt{Q}^{\bullet}(-1,t)+\mathtt{Q}^{\bullet}(-2,t) \sim \frac{1}{27} \sqrt{\frac{5}{\pi}} t^{-1/2}.
  	 \end{equation}
  	 
  By a first step analysis, we can write 
  \begin{equation}
  \tilde{\mathtt{Q}}^{\bullet}_{-2}(u) \sim \frac{1}{27} \tilde{\mathtt{Q}}^{021}_{0}(u)
  \end{equation}
  and similarly 
  \begin{equation}
  \tilde{\mathtt{Q}}^{\bullet}_{-1}(u) 	  \sim\frac{1}{27} \tilde{\mathtt{Q}}^{021}_{1}(u) + \frac{1}{9} \tilde{\mathtt{Q}}^{*01}_{0}(u)  +\frac{1}{9} \tilde{\mathtt{Q}}^{*12}_{0}(u)   + \frac{1}{27} \tilde{\mathtt{Q}}^{020}_{0}(u) + \frac{1}{27} \tilde{\mathtt{Q}}^{220}_{0}(u), 
  \end{equation} 
  where $*$ denotes an arbitrary element of $\mathbb{Z}_{3}$. In particular, it follows that the right hand sides of the last two asymptotic expressions diverge as $u\to 1^-$.

  Next, writing down the first asymptotic relation in Proposition \ref{prop:coefficient} for $x=(i,j,k)\in (\mathbb{Z}_{3})^{3}$, one notices that $\tilde{Q}^{ijk}_{r}(u)\sim \tilde{Q}^{i'j'k'}_{r}(u)$ if $k=k'\in \{0,1\}$ or $k=k'=2$ and $j=j'$. Applying the first relation in Proposition \ref{prop:coefficient} with $r\in \{0,1\}$, we find 
  \begin{eqnarray}\label{recursion_1'}
  \begin{cases}
  3\tilde{\mathtt{Q}}^{ ** 0}_{0}(u) \sim \tilde{\mathtt{Q}}^{ ** 0}_{0}(u)+\tilde{\mathtt{Q}}^{ ** 1}_{1}(u)\\
  3\tilde{\mathtt{Q}}^{ ** 1}_{0}(u)  \sim \tilde{\mathtt{Q}}^{ ** 1}_{0}(u)+\tilde{\mathtt{Q}}^{ *12}_{1}(u) \\
  3\tilde{\mathtt{Q}}^{ *02}_{0}(u) \sim \tilde{\mathtt{Q}}^{ ** 0}_{1}(u)+\tilde{\mathtt{Q}}^{ ** 1}_{2}(u)+\tilde{\mathtt{Q}}^{ *22}_{0}(u) \\
  3\tilde{\mathtt{Q}}^{ *12}_{0}(u) \sim \tilde{\mathtt{Q}}^{ *22}_{0}(u) \\
  3\tilde{\mathtt{Q}}^{ *22}_{0}(u) \sim \tilde{\mathtt{Q}}^{ ** 0}_{1}(u)+\tilde{\mathtt{Q}}^{ *22}_{0}(u) \\
  3\tilde{\mathtt{Q}}^{ **0}_{1}(u) \sim \tilde{\mathtt{Q}}^{ ** 0}_{1}(u)+\tilde{\mathtt{Q}}^{ ** 1}_{2}(u)+\tilde{\mathtt{Q}}^{ *02}_{0}(u)\\
  3\tilde{\mathtt{Q}}^{ *22}_{1}(u) \sim \tilde{\mathtt{Q}}^{ ** 0}_{1}(u)+\tilde{\mathtt{Q}}^{ ** 1}_{0}(u)+\tilde{\mathtt{Q}}^{ *22}_{1}(u).
  \end{cases}
  \end{eqnarray}
  On the other hand, the second asymptotic relation in Proposition \ref{prop:coefficient} gives 
 \begin{equation}
 \begin{cases}
 \tilde{\mathtt{Q}}^{** 1}_{1}(u)\sim \tilde{\mathtt{Q}}^{** 1}_{0}(u)+\tilde{\mathtt{Q}}^{** 0}_{0}(u)\\
 \tilde{\mathtt{Q}}^{ *12}_{1}(u)\sim \tilde{\mathtt{Q}}^{ *12}_{0}(u)+\tilde{\mathtt{Q}}^{** 1}_{0}(u).
 \end{cases}
 \end{equation}
 Combining, these give us 
 \begin{equation}
	\frac{1}{2} \tilde{\mathtt{Q}}^{**1}_{1}(u) \sim \tilde{\mathtt{Q}}^{** 1}_{0}(u)\sim \tilde{\mathtt{Q}}^{**0}_{0}(u)\sim \tilde{\mathtt{Q}}^{*12}_{0}(u).
 \end{equation}
 Hence 
 \begin{equation}
	 \mathtt{Q}^{\bullet}(-1,t)+\mathtt{Q}^{\bullet}(-2,t) \sim \left(\frac{4}{27} + \frac{2}{27} + \frac{1}{9} + \frac{2}{27}\right)\mathtt{Q}^{**1}(0,t),
 \end{equation}
 so we obtain 
 \begin{equation}
	 \mathtt{Q}^{**1}(0,t)\sim \sqrt{\frac{5}{121\pi}} t^{-1/2}.
 \end{equation}
 One can determine the asymptotics of other survival probabilities in a similar way. $\blacktriangle$
  \end{customexample}

  \vspace{0.5cm}
  \section{Application: The 3-color excitable media on $\mathbb{Z}$}
  \label{section:application}
  
In this section, we apply our theory of survival probabilities for functionals of Markovian increments to analyze clustering behavior in three cellular automata models for excitable media on the  one dimensional integer lattice $\mathbb{Z}$: the \textit{cyclic cellular automaton} (CCA), the \textit{Greenberg-Hastings model} (GHM), and the \textit{firefly cellular automaton} (FCA). In general, given a simple graph $G=(V,E)$ and an integer parameter $\kappa\ge 3$ as the number of available colors per site, they are discrete-time dynamical systems whose trajectories are given by a sequence of $\kappa$-colorings $(X_{t})_{t\ge 0}$, $X_{t}:V\rightarrow \mathbb{Z}_{\kappa}$, with a deterministic and parallel update rule $X_{t}\mapsto X_{t+1}$. 

 \begin{figure*}[h]
	\centering
	\includegraphics[width=1 \linewidth]{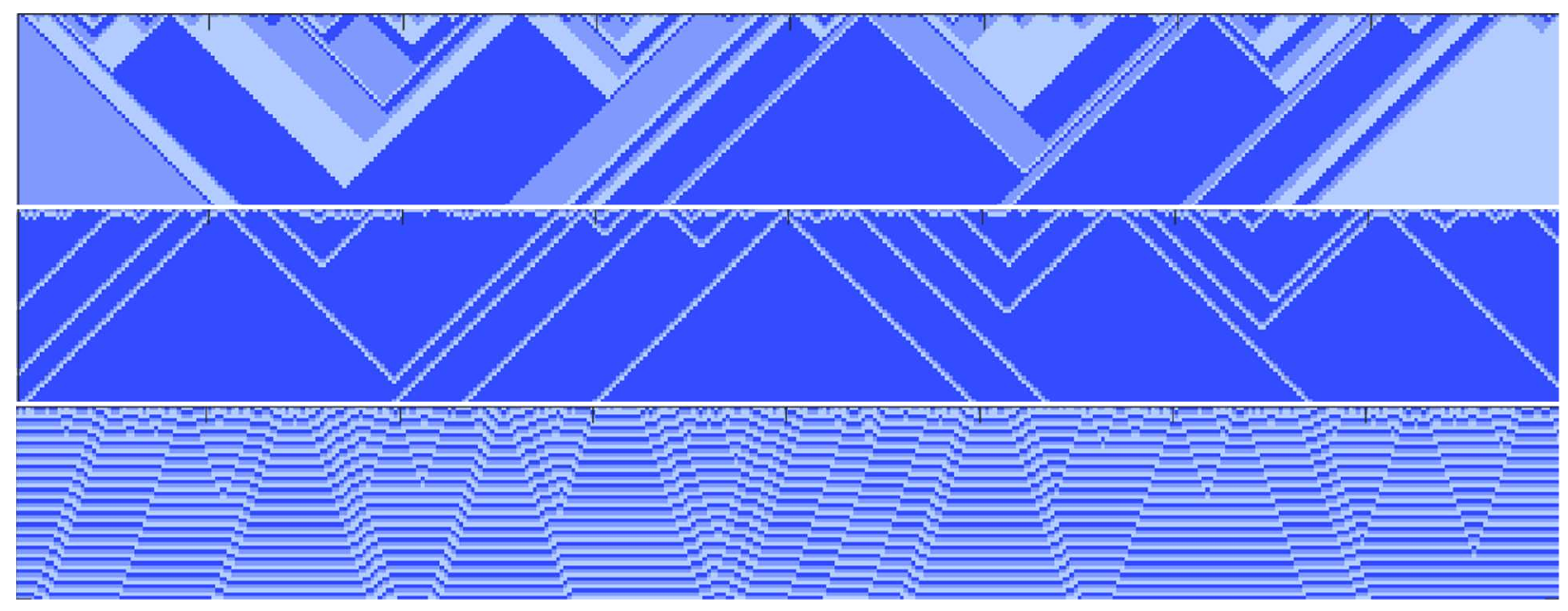}
	\vspace{-0.3cm}
	\caption{ Simulation of $3$-color CCA (top), GHM (middle), and FCA (bottom) on a path of 400 nodes for 50 iterations. The top rows are a random $3$-color initial coloring, and a single iteration of the corresponding transition map generates each successive row from top to bottom. Dark blue=0, blue=1, and light blue=2.
	}
	\label{3FCA_ex}
\end{figure*}

In case of $\kappa=3$ and $G=\mathbb{Z}$, all three models mentioned above share similar annihilating embedded particle system structures, which allow us to connect their limiting behavior to the survival of an associated random walk (see Figure \ref{3FCA_ex}). More precisely, when the initial 3-coloring $X_{0}$ is chosen at random from the uniform product measure on $(\mathbb{Z}_{3})^{\mathbb{Z}}$, for each model, there exists a constant $C>0$ such that for all $x\in \mathbb{Z}$, 
   \vspace{0.1cm}
  \begin{equation}\label{eq:excitable_asymptotics}
	  \mathbb{P}(X_{t}(x)\ne X_{t}(x+1)) \sim C t^{-1/2}.
	  \vspace{0.1cm}
  \end{equation}
  Fisch \cite{fisch1992clustering} proved \eqref{eq:excitable_asymptotics} for the CCA with $C=\sqrt{2/3\pi}$, and Durrett and Steif \cite{durrett1991some} proved \eqref{eq:excitable_asymptotics} for the GHM with $C=\sqrt{2/27\pi}$. We recover these results and prove \eqref{eq:excitable_asymptotics} for the FCA with $C=\sqrt{8/81\pi}$, hence establishing Theorem \ref{thm:3FCA_Z}.

  \subsection{The 3-color CCA on $\mathbb{Z}$.}
  
  The CCA was introduced by Bramson and Griffeath \cite{bramson1989flux} as a discrete time analogue of the cyclic particle systems. In this model, we imagine each vertex of $\Z$ is inhabited by one of $3$ different species in a cyclic food chain, and at each step species of color $i$ are eaten (and thus replaced) by a neighboring species of color $(i+1)$ mod $3$, if one exists. More precisely, the time evolution is governed by 
  \begin{equation*}
  \text{(CCA)}\qquad X_{t+1}(v)=\begin{cases}
  X_{t}(v)+1 \,\, (\text{mod $\kappa$})& \text{if $\exists u\in\{v-1,v+1\}$ s.t. $X_{t}(u)=X_{t}(v)+1$ (mod $3$) }\\
  X_{t}(v) & \text{otherwise}.
  \end{cases}
  \end{equation*}
  
The $3$-color CCA on $\Z$ was studied extensively by Fisch \cite{fisch1992clustering} using a connection to survival of an associated simple random walk, which arises by considering an associated annihilating particle system. Namely, for a given $3$-coloring $X:\mathbb{Z}\rightarrow \mathbb{Z}_{3}$ and $x\in \mathbb{Z}$, we define its \textit{color differential} $dX(x,x+1)\in \{-1,0,1\}$ by 
  \begin{equation}
	  dX(x,x+1) = X(x+1)-X(x) \mod 3. 
  \end{equation}
  Then for each $t\ge 0$, we place an $\R$ (resp., $\L$) particle on every edge $(x,x+1)$ with $dX_{t}(x,x+1)=-1$ (resp., $1$), and leave all the other edges unoccupied. The CCA dynamics induce the dynamics on these particles, which has the following simple description: each $\R$ (resp., $\L$) particle moves to its right (resp., left) with unit speed, and if two opposing particles ever cross or have to occupy the same edge, then they annihilate each other. Hence if one wants to have an $\R$ particle on the edge $(0,1)$ at time $t$, which is the event on the left hand side of (\ref{eq:3CCA_survival}), then this $\R$ particle must have been on the edge $(-t,-t+1)$ at time $0$ and it must travel distance $t$ without being annihilated by an opposing particle. Thus, on the edges of every interval $[-t,s]$ for $-t+1\le s \le t+1$, there must be more $\R$ than $\L$ particles. The converse of this observation is also true, which gives the following (this is Proposition $5.3$ in~\cite{fisch1992clustering}).
  \begin{prop}\label{prop:survival_Z}
  	Let $(X_{t})_{t\ge 0}$ be a 3-color CCA trajectory on $\mathbb{Z}$. Then 
  	\begin{equation}\label{eq:3CCA_survival}
  	\left\{ dX_{t}(0,1)=-1  \right\} = \left\{ \sum_{x=-t}^{s} dX_{0}(x,x+1)\le -1 \text{ for all } -t\le s \le t \right\}.
  	\end{equation}  
  \end{prop}
   
  Now if $X_{0}$ is drawn from the uniform product measure on $(\mathbb{Z}_{3})^{\mathbb{Z}}$, then the initial edge increment $dX_{0}(x,x+1)$ is uniform on $\{-1,0,1\}$ and independent for all $x\in \mathbb{Z}$. By letting $\mathtt{X}_n = \mathtt{X}_{n;t} = -dX_{0}(n-1-t,n-t)$ and $g(\mathtt{X}_n) = \mathtt{X}_n$ for $n\ge 1$ and $t\ge 0$, we get 
  \begin{equation*}
	\gamma_{g}^{2} = 2/3.
  \end{equation*}
  By Proposition \ref{prop:survival_Z}, translation invariance, and Theorem \ref{thm:general_survival_Brownian}, we obtain
  \begin{equation*}
  \begin{aligned}
	  \mathbb{P}(X_{t}(x)\ne X_{t}(x+1))&=2\ \mathbb{P}(dX_{t}(0,1)=-1)\\
	  &= 2\ \mathbb{P}\left(\sum_{n=1}^{s}g(\mathtt{X}_n) \ge 1 \text{ for all }  1\le s \le 2t+1 \right)\\
	  &= 2\ \mathtt{Q}^{\bullet}(-1,2t+1)\\
	  &\sim \sqrt{\frac{2}{3\pi}}t^{-1/2}. 
	  \end{aligned}
  \end{equation*}
  
We remark here that our inspiration for Lemma \ref{lemma:key} came from considering an alternative way of analyzing the CCA. Consider a 3-color CCA trajectory $(X_{t})_{t\ge 0}$ on $\mathbb{N}_{0}$ (rather than on all of $\Z$). For each $x\in \mathbb{N}_{0}$ and $t\ge 1$, define
  \begin{equation}
	  \nbe_{t}(x) = \sum_{s=0}^{t-1}\1(X_{s}(x)\ne X_{s+1}(x)),
  \end{equation} 
  which counts the number of `excitations' of $x$ during the first $t$ iterations. In terms of the annihilating particle system, $\nbe_{t}(0)$ counts the number of $\L$ particles that reach the origin during the first $t$ iterations. Hence for all $t\ge 1$, we have 
  \begin{equation}\label{eq:nbe}
	  \nbe_{t}(0) - \nbe_{t-1}(0) = \1(dX_{t}(0,1)=1).
  \end{equation}
  By the analogue of Proposition \ref{prop:survival_Z} for $G=\mathbb{N}_{0}$ and an $\L$ particle at the origin, we have
 \begin{equation}\label{eq:3CCA_survival2}
  \{dX_{t}(0,1)=1\} = \left\{ \sum_{x=s}^{t} dX_{0}(x,x+1)\ge 1 \text{ for all } 0\le s \le t \right\}.
  \end{equation}
  That is, the event that there is an $\L$ particle at the origin at time $t$ is equivalent to the survival of the associated walk up to $t+1$ steps.
  
  On the other hand, the quantity $\nbe_{t}(0)$ is closely related to the maximum of an associated walk. The first connection between the 3-color CCA dynamics on $\mathbb{Z}$ and the maximum of associated simple random walk appears in the work of Belitsky and Ferrari \cite{belitsky1995ballistic}. In a recent joint work with Gravner \cite{gravner2016limiting}, we determined the limiting behaviors of the $3$-color CCA and GHM on arbitrary underlying graphs by making use of a so-called ``tournament expansion''.  Lemma $1$ in \cite{gravner2016limiting} states that $\nbe_{t}(x)$ on any graph $G=(V,E)$ can be expressed as the maximum of partial sums of time-0 edge increments among all walks of length $t$ starting from $x$. In particular, for $G=\mathbb{N}_{0}$ and $x=0$, it says that
  \begin{equation}\label{eq:3CCA_max}
	  \nbe_{t}(0) = \max_{0\le s \le t-1} \sum_{0\le x \le s} dX_{0}(x,x+1). 
  \end{equation}
  Combining \eqref{eq:nbe},\eqref{eq:3CCA_survival2} and~\eqref{eq:3CCA_max} gives a special case of Lemma \ref{lemma:key} for functionals $g$ taking values in $\{-1,0,1\}$.
  
   \vspace{0.2cm}
  \subsection{The 3-color GHM on $\mathbb{Z}$.}
  The GHM was introduced by Greenberg and Hastings \cite{greenberg1978spatial} to capture the phenomenological essence of neural networks in a discrete setting. Its time evolution rule is given by 
  \begin{equation*}
  (\text{GHM})\qquad X_{t+1}(v)=\begin{cases}
  1 & \text{if $X_{t}(v)=0$ and $\exists u\in N(v)$ s.t. $X_{t}(u)=1$  }\\
  0 & \text{if $X_{t}(v)=0$ and $\nexists u\in N(v)$ s.t. $X_{t}(u)= 1$  }\\
  X_{t}(v)+1\,\,\text{(mod $3$)} & \text{otherwise}.
  \end{cases}
  \end{equation*}
  In words, sites of color $0$ (rested) remain in color $0$ unless they get excited by a neighboring site of color $1$ (excited); sites of colors $1$ or $2$ (refractory) increment their colors by $1$ (mod~$3$).
  
  Consider the 3-color GHM trajectory $(X_{t})_{t\ge 0}$ on $\mathbb{Z}$, where $X_{0}$ is drawn from the uniform product measure. Since in GHM sites of color 1 excite sites of color 0, we define 
  \begin{equation*}
  dX_{0}(x,x+1) = 
  \begin{cases}
  1 & \text{if $X_0(x) = 0$ and $X_0(x+1) = 1$}\\
  -1 & \text{if $X_0(x) = 1$ and $X_0(x+1) = 0$}\\
  0 & \text{otherwise.}
  \end{cases}
  \end{equation*}
  Like the $3$-color CCA, the GHM gives rise to an annihilating particle system: for $t\ge 0$, we place an $\R$ (resp., $\L$) particle on every edge $(x,x+1)$ with $dX_{t}(x,x+1)=-1$ (resp., $1$), and leave all the other edges unoccupied. The GHM dynamics induce the same dynamics on these particles as the CCA: each $\R$ (resp., $\L$) particle moves to its right (resp., left) with unit speed, and if two opposing particles ever cross or have to occupy the same edge, then they annihilate each other. Since the particle system dynamics are the same, it is not hard to see that Proposition \ref{prop:survival_Z} still holds for GHM. However, we remark that the increments $(dX_0(x,x+1))_{x\in\Z}$ are not Markovian in the case of GHM. Indeed, given that $dX_{0}(x,x+1)=0$, the events $\{dX_{0}(x+1,x+2)=1\}$ and $\{dX_{0}(x-1,x)=1\}$ are mutually exclusive. Occurrence of both events requires $(X_{0}(x-1),X_{0}(x),X_{0}(x+1),X_{0}(x+2))=(0,1,0,1)$, which implies $dX_{0}(x,x+1)=-1$, a contradiction.
  
  In order to apply our main result, we can view $dX_{0}(x,x+1)$ as a functional $g: (\mathbb{Z}_{3})^{2}\rightarrow \{-1,0,1\}$ on a stationary Markov chain $(\mathtt{X}_{n})_{n\ge1}$. That is, to apply Proposition \ref{prop:survival_Z}, we let $\mathtt{X}_{n}=\mathtt{X}_{n;t}:=(X_{0}(n-t-1),X_{0}(n-t))$ and let $g(\mathtt{X}_n) := -dX_0(n-t-1,n-t)$ for $n\ge 1$ and $t\ge 0$. By the independence of $\mathtt{X}_{n}$ and $\mathtt{X}_{n+k}$ for all $n\ge 1$ and $k\ge 2$, it is easy to compute 
  \begin{equation*}
  \gamma_{g}^{2} = \mathbb{E}(g(\mathtt{X}_{1})^{2}) + 2\mathbb{E}(g(\mathtt{X}_{1})g(\mathtt{X}_{2})) = \frac{2}{9}-\frac{4}{27}=\frac{2}{27}. 
  \end{equation*}
  Hence Theorem \ref{thm:general_survival_Brownian} recovers the desired asymptotic 
  \begin{equation}\label{eq:3GHM_asymptotic}
  \mathbb{P}(X_{t}(x)\ne X_{t}(x+1))=2\mathtt{Q}^{\bullet}(-1,2t+1) \sim \sqrt{\frac{2}{27\pi}} t^{-1/2}.
  \end{equation}

  \vspace{0.2cm}
  \subsection{The 3-color FCA on $\mathbb{Z}$.}
  
  In this subsection we prove Theorem \ref{thm:3FCA_Z}. The FCA was introduced recently by the first author as a discrete model for pulse-coupled oscillators \cite{lyu2015synchronization}. Here, one envisions each site as a $3$-state oscillator which `blinks' whenever it has the designated `blinking color' $1$. Sites increment their colors by $1$ (mod $3$) unless they have the post-blinking color $2$ and have a blinking neighbor (of color $1$), in which case they do not update their colors (they wait for the blinking neighbor to catch up). The precise time-evolution rule is given by (\ref{FCA_transition_rule}) in the introduction.
  
  As before, consider the 3-color FCA trajectory $(X_{t})_{t\ge 0}$ on $\mathbb{Z}$ with random initial coloring $X_{0}$ drawn from the uniform product measure. As with the CCA, for each $t\ge 0$ and $x\in \mathbb{Z}$, we define $dX_{t}(x,x+1)=X_{t}(x+1)-X_{t}(x)\mod 3\in \{-1,0,1\}$. By putting an $\R$ (resp. $\L$) particle on each edge $(x,x+1)$ with $dX_{t}(x,x+1)=1$ (resp. $-1$; note that this is the reverse of CCA and GHM), the FCA induces an edge particle dynamics $(dX_{t})_{t\ge 0}$. Two features of the particle dynamics induced by the 3-color FCA are that some particles may `flip' their directions during the first iteration $X_{0}\mapsto X_{1}$, and that each particle moves with speed $1/3$ after time $t=1$.  Particles move only when their `tails' are blinking: an $\R$  (resp. $\L$) particle at $(x,x+1)$ can only move right (resp. left) if $x$ (resp. $x+1$) is in state $1$. After time $t=1$, particles annihilate when they meet or cross, as with GHM and CCA. See Figure \ref{3FCA_ex} for an illustration of these dynamics.

  \begin{figure*}[h]
  	\centering
  	\includegraphics[width=0.5 \linewidth]{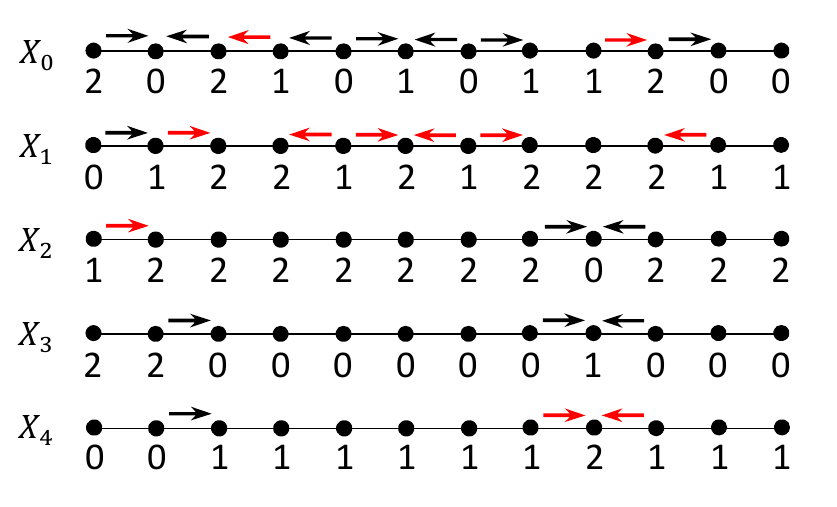}
  	\vspace{-0.3cm}
  	\caption{ Example of a 3-color FCA trajectory on a finite path. Only edge particles that have color 1 at their tails (red arrows) move at each iteration. Initial configurations of particles corresponding to color triples of $120$ and $021$ may flip their direction during the first iteration. Thereafter, particles move in their given directions with speed 1/3 until being annihilated by opposing particles.
  	}
  	\label{3FCA_ex}
  \end{figure*}

  Note that the edge particle dynamics of FCA on the time scale $3t+1$ match those of GHM and CCA (without rescaling time) after the initial step. This is formalized by the following proposition, which is the FCA counterpart of Proposition \ref{prop:survival_Z}, whose proof we give at the end of this subsection. 
  
  \begin{prop}\label{prop:FCA_survival_Z}
  	Let $(X_{t})_{t\ge 0}$ be a 3-color FCA trajectory on $\mathbb{Z}$. Then 
  	\begin{equation}\label{eq:FCA_survival}
  	\left\{ dX_{3t+1}(0,1)=1  \right\} = \left\{ \sum_{x=-t}^{s} dX_{1}(x,x+1) \ge 1  \quad \text{for all $-t\le s \le t$}   \right\}.
  	\end{equation}  
  \end{prop}

  Notice that the event on the right in \eqref{eq:FCA_survival} uses edge increments $dX_{1}$ instead of the initial ones $dX_{0}$; this is to avoid the issue of particles changing their direction in the first step. 

As for the 3-color FCA, the increments $dX_{1}(x,x+1)$ are not Markovian. In fact, they have infinite-range correlation, contrary to the CCA or GHM. To see this, we first observe that no three consecutive sites have color triples $(1,2,0)$ or $(0,2,1)$ after time $t=1$.
  
  \begin{prop}\label{no_flips}
  	Let $(X_{t})_{t\ge 0}$ be a 3-color FCA trajectory on $\mathbb{Z}$. Then for all $t\ge 1$ and $x\in \mathbb{Z}$, 
  	\begin{equation*}
  	(X_{t}(x-1),X_{t}(x),X_{t}(x+1))\notin \{(1,2,0), (0,2,1)\}.
  	\end{equation*}
  \end{prop}
  
  \begin{proof}
  	Suppose for the sake of contradiction that $(X_{t}(x-1),X_{t}(x),X_{t}(x+1))=(1,2,0)$ for some $x\in \mathbb{Z}$ and $t\ge 1$. Then we must have $X_{t-1}(x-1)=0$ and $X_{t-1}(x+1)=2$ and  $X_{t-1}(x)\in \{1,2\}$. If $X_{t-1}(x)=1$, then $x$ is blinking, which causes $x+1$ to not update its state, so $X_{t}(x+1)=2\ne 0$, a contradiction. If $X_{t-1}(x)=2$, then $X_{t}(x)=0\neq 2$, since neither $x-1$ nor $x+1$ are blinking at time $t-1$; this is also a contradiction. This shows the assertion. 
  \end{proof}
  
  Now the color patterns `120' and `021' are prohibited at time 1. For any $x\ge 0$, the event $dX_{1}(x,x+1)=1$ is dependent on the event $dX_{1}(-1,0)=1$: with positive probability we have $X_{1}(0)=2$, $X_{1}\equiv 0$ on $[1,x)$, and $X_{1}(x)=1$; then $dX_{1}(-1,0)=1$ would imply that we have a `120' pattern on $[-1,1]$ at time 1, which is impossible. Hence the time 1 increments $dX_{1}(x,x+1)$ have long-range correlation.
  
  Still, we can handle this long-range correlation by viewing the increments as given by a functional of the underlying stationary Markov chain of color quadruples at time 0. 
  
  \vspace{0.2cm}
  \hspace{-0.42cm}\textbf{Proof of Theorem \ref{thm:3FCA_Z}.} For each $x\in \mathbb{Z}$, $dX_{1}(x,x+1)$ is a function of the initial colors on the four sites in the interval $[x-1,x+2]$. So for each fixed $t\ge 0$, we may take the underlying chain $\mathtt{X}_{n}=\mathtt{X}_{n;t}=(X_{0}(n-2-t),X_{0}(n-1-t),X_{0}(n-t),X_{0}(n+1-t))$ for $n\ge 1$, and define $g:(\mathbb{Z}_{3})^{4}\to\{-1,0,1\}$ by updating a given coloring on the path of four nodes according to the FCA rule and taking color differential of the middle two sites. That is,
  $$
  g(\mathtt{X}_n) = X_1(n-t) - X_1(n-1-t) \quad \text{ for $n\ge 1$.}
  $$

  
  
  Since $X_{0}$ is drawn from the product measure, two increments from disjoint quadruples are independent. By a direct calculation or using Proposition \ref{prop:gamma_formula}, we can compute 
  \begin{equation*}
	  \gamma_{g}^{2} = \frac{40}{81} + 2\left( -\frac{17}{243} - \frac{19}{729} - \frac{2}{729} \right) = \frac{8}{27}.
  \end{equation*}  
  Hence Theorem \ref{thm:general_survival_Brownian} yields 
  \begin{equation*}
	  \mathtt{Q}^{\bullet}(-1,2t+1) \sim \sqrt{\frac{2}{27\pi}} t^{-1/2}.
	   \vspace{0.1cm}
  \end{equation*} 
  Furthermore, we note that the probability $\mathbb{P}(X_{t}(x)\ne X_{t}(x+1))$ is the edge particle density on $\mathbb{Z}$ at time $t$, which is independent of the location $x\in \mathbb{Z}$. Thus by Proposition \ref{prop:FCA_survival_Z}, we have 
  \begin{equation}\label{eq:3FCA_Z_asymptotic}
  \mathbb{P}(X_{3t+1}(x)\ne X_{3t+1}(x+1)) = 2\ \mathtt{Q}^{\bullet}(-1,2t+1) \sim \sqrt{\frac{8}{27\pi}} t^{-1/2}.
  \end{equation} 
  
  To finish, we observe that $\mathbb{P}(X_{t}(x)\ne P_{t}(x+1))$ is independent of $x\in \mathbb{Z}$ by translation invariance of the initial measure and the transition map. Denote this probability by $p_{t}$. Moreover, the coloring at time $t$, $(X_t(x))_{x\in\Z}$, is a $2t$-dependent stationary sequence, which implies the sequence of indicators $(\1(X_t(x)\ne X_t(x+1)))_{x\in \Z}$ is mixing and therefore ergodic. Hence, by Birkhoff's ergodic theorem, we have 
  \begin{equation*}
 	p_{t} = \lim_{M\rightarrow \infty} \frac{\text{$\#$ of particles on $[-M,M]$ at time $t$ }}{2M} \quad \text{a.s..}
  \end{equation*}  
  The fact that particles can only move at most $1$ space or annihilate in a single step (after time $t=1$) implies that $p_{t}$ is non-increasing in $t\ge 1$, since
    \begin{equation*}
    \begin{aligned}
 	p_{t+1} &= \lim_{M\rightarrow \infty} \frac{\text{$\#$ of particles on $[-M,M]$ at time $t+1$ }}{2M} \\
	&\le \lim_{M\rightarrow \infty} \frac{2+ \text{$\#$ of particles on $[-M,M]$ at time $t$ }}{2M}\\
	&= p_t.
     \end{aligned}
  \end{equation*} 
   This yields $p_{3t}\sim p_{3t+1}\sim p_{3t+2}$, so we obtain the assertion by a time change in \eqref{eq:3FCA_Z_asymptotic}. $\blacksquare$

  \vspace{0.2cm}
  We finish this section by giving a proof of Proposition \ref{prop:FCA_survival_Z}.
  
  \vspace{0.1cm}
  \hspace{-0.42cm}\textbf{Proof of Proposition \ref{prop:FCA_survival_Z}.} Recall the definition of color differential $dX_{t}(x,x+1)$ for the 3-color FCA given at the beginning of the subsection. At time $t=1$, on each edge $(x,x+1)$ with $dX_{1}(x,x+1)=1$ (resp. $-1$) we put an $\R$ particle (resp. $\L$), and otherwise leave the edge unoccupied. We show that the particle dynamics induced by $(dX_{t})_{t\ge 1}$ is so that all $\R$ and $\L$ particles move with constant speed $1/3$ without changing their direction, and two opposing particles that have to occupy the same edge or cross each other annihilate. By symmetry, consider an $\R$ particle on the edge $(x,x+1)$ at time $t\ge 1$. We show that
	\begin{description}[noitemsep]
		\item{(i)} If $X_{t}(x)\ne 1$, then the $\R$ particle is on the same edge $(x,x+1)$ at time $t+1$;
		\vspace{0.1cm}
		\item{(ii)} If $X_{t}(x)= 1$ and there is an $\L$ particle on the edge $(x+1,x+2)$ or $(x+2,x+3)$ at time $t$, then the $\R$ particle is annihilated at time $t+1$; otherwise, it is on the edge $(x+1,x+2)$ at time $t+1$. 
		\vspace{0.1cm}
		\item{(iii)} If the $\R$ particle jumps to the edge $(x+1,x+2)$ during $X_{t}\mapsto X_{t+1}$, then it stays on that edge through time $t+3$, and either gets annihilated or jumps to the edge $(x+2,x+3)$ during $X_{t+3}\mapsto X_{t+4}$. 
	\end{description}

	First suppose $(X_{t}(x),X_{t}(x+1))=(0,1)$. Then $(X_{t+1}(x),X_{t+1}(x+1))=(1,2)$, so the $\R$ particle stays at the same edge without being annihilated. If  $(X_{t}(x),X_{t}(x+1))=(2,0)$, then since $X_{t}(x-1)\ne 1$ by Proposition \ref{no_flips}, we have $(X_{t+1}(x),X_{t+1}(x+1))=(0,1)$, so the $\R$ particle is on the same edge $(x,x+1)$ at time $t+1$. This shows (i). To show (ii), suppose $(X_{t}(x),X_{t}(x+1))=(1,2)$. By Proposition \ref{no_flips}, $X_{t}(2)\in \{1,2\}$. If $X_{t}(2)=1$, then $(X_{t+1}(0),X_{t+1}(1),X_{t+1}(2))=(2,2,2)$ so the $\R$ (and opposing $\L$) particle is annihilated during $X_{t}\mapsto X_{t+1}$. If $X_{t}(2)=2$ and $X_{t}(3)=1$, then $(X_{t+1}(0),X_{t+1}(1),X_{t+1}(2),X_{t+1}(3))=(2,2,2,2)$ so the similar conclusion holds. Otherwise, $X_{t}(3)\ne 1$, so we have  $(X_{t+1}(0),X_{t+1}(1),X_{t+1}(2),X_{t+1}(3))=(2,2,0,*)$ so the $\R$ particle is on the edge $(1,2)$ at time $t+1$. This shows (ii). Lastly, to show (iii), suppose the $\R$ particle jumps to the edge $(x+1,x+2)$ during $X_{t}\mapsto X_{t+1}$. By the previous cases and Proposition \ref{no_flips},  $(X_{t+1}(x),X_{t+1}(x+1),X_{t+1}(x+2))=(*,2,0)$ where $*\ne 1$. So $(X_{t+2}(x+1),X_{t+2}(x+2))=(0,1)$ and $(X_{t+3}(x+1),X_{t+3}(x+2))=(1,2)$. Hence (iii) follows from (i) and (ii).

	Now we are ready to prove the assertion. Suppose $dX_{3t+1}(0,1)=1$ for some $t\ge 0$. By (iii), the $\R$ particle on the edge $(0,1)$ at time $3t+1$ must have been on the edge $(-t,-t+1)$ at time $1$. By (i)-(iii), this $\R$ particle reaches the edge $(0,1)$ at time $3t+1$ if and only if there are strictly more $\R$ particles than $\L$ particles at time $1$ on the edges in the interval $[-t,-t+s]$ for all $1\le s \le 2t+1$. This is exactly the event in the right hand side of the assertion. This shows the assertion. $\blacksquare$

	\vspace{0.4cm}
	\section*{Acknowledgements} Hanbaek Lyu was partially supported by the university Presidential Fellowship. David Sivakoff was partially supported by NSF grant DMS--1418265.

	\small{
	\bibliographystyle{plain}   
	\bibliography{mybib}  

\begin{thebibliography}{10}

\bibitem{andersen1953sums}
Erik~Sparre Andersen.
\newblock On sums of symmetrically dependent random variables.
\newblock {\em Scandinavian Actuarial Journal}, 1953(sup1):123--138, 1953.

\bibitem{arjas1973symmetric}
Elja Arjas and TP~Speed.
\newblock Symmetric wiener-hopf factorisations in markov additive processes.
\newblock {\em Probability Theory and Related Fields}, 26(2):105--118, 1973.

\bibitem{aurzada2015persistence}
Frank Aurzada and Thomas Simon.
\newblock Persistence probabilities and exponents.
\newblock In {\em L{\'e}vy Matters V}, pages 183--224. Springer, 2015.

\bibitem{belitsky1995ballistic}
Vladimir Belitsky and Pablo~A Ferrari.
\newblock Ballistic annihilation and deterministic surface growth.
\newblock {\em Journal of statistical physics}, 80(3-4):517--543, 1995.

\bibitem{blumenthal1964additive}
RM~Blumenthal and RK~Getoor.
\newblock Additive functionals of markov processes in duality.
\newblock {\em Transactions of the American Mathematical Society},
  112(1):131--163, 1964.

\bibitem{bramson1989flux}
Maury Bramson and David Griffeath.
\newblock Flux and fixation in cyclic particle systems.
\newblock {\em The Annals of Probability}, pages 26--45, 1989.

\bibitem{bray2013persistence}
Alan~J Bray, Satya~N Majumdar, and Gr{\'e}gory Schehr.
\newblock Persistence and first-passage properties in nonequilibrium systems.
\newblock {\em Advances in Physics}, 62(3):225--361, 2013.

\bibitem{dedecker2000functional}
J{\'e}r{\^o}me Dedecker and Emmanuel Rio.
\newblock On the functional central limit theorem for stationary processes.
\newblock In {\em Annales de l'IHP Probabilit{\'e}s et statistiques},
  volume~36, pages 1--34, 2000.

\bibitem{ding2011equality}
Jiu Ding and Noah~H Rhee.
\newblock On the equality of algebraic and geometric multiplicities of matrix
  eigenvalues.
\newblock {\em Applied Mathematics Letters}, 24(12):2211--2215, 2011.

\bibitem{durrett1991some}
Richard Durrett and Jeffrey~E Steif.
\newblock Some rigorous results for the greenberg-hastings model.
\newblock {\em Journal of Theoretical Probability}, 4(4):669--690, 1991.

\bibitem{feller1971introduction}
William Feller.
\newblock An introduction to probability and its applications, vol. ii.
\newblock {\em Wiley, New York}, 1971.

\bibitem{fisch1992clustering}
Robert Fisch.
\newblock Clustering in the one-dimensional three-color cyclic cellular
  automaton.
\newblock {\em The Annals of Probability}, pages 1528--1548, 1992.

\bibitem{gravner2016limiting}
Janko Gravner, Hanbaek Lyu, and David Sivakoff.
\newblock Limiting behavior of 3-color excitable media on arbitrary graphs.
\newblock {\em arXiv preprint arXiv:1610.07320}, 2016.

\bibitem{greenberg1978spatial}
James~M Greenberg and SP~Hastings.
\newblock Spatial patterns for discrete models of diffusion in excitable media.
\newblock {\em SIAM Journal on Applied Mathematics}, 34(3):515--523, 1978.

\bibitem{ivanovs2017splitting}
Jevgenijs Ivanovs.
\newblock Splitting and time reversal for markov additive processes.
\newblock {\em Stochastic Processes and their Applications}, 127(8):2699--2724,
  2017.

\bibitem{kac1974stochastic}
Marc Kac.
\newblock A stochastic model related to the telegrapher's equation.
\newblock {\em Rocky Mountain Journal of Mathematics}, 4(3), 1974.

\bibitem{klusik2014note}
Przemys{\l}aw Klusik and Zbigniew Palmowski.
\newblock A note on wiener--hopf factorization for markov additive processes.
\newblock {\em Journal of Theoretical Probability}, 27(1):202--219, 2014.

\bibitem{latouche1999introduction}
Guy Latouche and Vaidyanathan Ramaswami.
\newblock {\em Introduction to matrix analytic methods in stochastic modeling}.
\newblock SIAM, 1999.

\bibitem{lyu2015synchronization}
Hanbaek Lyu.
\newblock Synchronization of finite-state pulse-coupled oscillators.
\newblock {\em Physica D: Nonlinear Phenomena}, 303:28--38, 2015.

\bibitem{meyn2012markov}
Sean~P Meyn and Richard~L Tweedie.
\newblock {\em Markov chains and stochastic stability}.
\newblock Springer Science \& Business Media, 2012.

\bibitem{presman1969factorization}
{\`E}~L Presman.
\newblock Factorization methods and boundary problems for sums of random
  variables given on markov chains.
\newblock {\em Mathematics of the USSR-Izvestiya}, 3(4):815, 1969.

\bibitem{spitzer1956combinatorial}
Frank Spitzer.
\newblock A combinatorial lemma and its application to probability theory.
\newblock {\em Transactions of the American Mathematical Society},
  82(2):323--339, 1956.

\bibitem{van2009random}
Remco Van Der~Hofstad.
\newblock Random graphs and complex networks.
\newblock {\em Available on http://www. win. tue. nl/rhofstad/NotesRGCN. pdf},
  page~11, 2009.

\bibitem{van2008elementary}
Remco Van Der~Hofstad and Michael Keane.
\newblock An elementary proof of the hitting time theorem.
\newblock {\em The American Mathematical Monthly}, 115(8):753--756, 2008.

\end{thebibliography}
}
	
	
	
	

\end{document}